\documentclass[11pt]{amsart}
\usepackage[utf8]{inputenc}
\usepackage[english]{babel}
\usepackage{multirow}
\usepackage{anysize}
\usepackage{color}
\usepackage{multicol}
\usepackage{tikz-cd}

\usepackage{enumitem}
\usepackage{amsfonts}
\usepackage{amsmath}
\usepackage{amssymb}
\usepackage{amsthm}

\usepackage[usenames,dvipsnames]{pstricks}
 \usepackage{pstricks-add}
 \usepackage{epsfig}
\usepackage{mathrsfs}

\newtheorem{Th}{Theorem}[section]

\newtheorem{Ob}[Th]{Remark}
\newtheorem{Cor}[Th]{Corollary}
\newtheorem{Ej}[Th]{Example}
\newtheorem{Pro}[Th]{Proposition}

\newtheorem{Le}[Th]{Lemma}

\title[Distributions and Hamiltonians on ${\rm SL}(n,{\mathbb R})$ induced by its Killing form]{Left-invariant distributions and metric Hamiltonians on ${\rm SL}(n,{\mathbb R})$ induced by its Killing form}

\author[Abraham Bobadilla Osses and Mauricio Godoy Molina]{Abraham Bobadilla Osses$^{1,2}$ and Mauricio Godoy Molina$^2$}

\address{$^1$ Facultad de Ingenier\'ia, Universidad Aut\'onoma de Chile – Sede Temuco, Av. Alemania 1090 Temuco, Chile.}
\email{abraham.bobadilla@uautonoma.cl}

\address{$^2$ Departamento de Matem\'atica y Estad\'istica, Universidad de La Frontera, Av. Francisco Salazar 01145 Temuco, Chile.}
\email{abraham.bobadilla@ufrontera.cl, mauricio.godoy@ufrontera.cl}



\begin{document}
\maketitle
\begin{abstract}
From the classical theory of Lie algebras, it is well-known that the bilinear form $B(X,Y)={\rm tr}(XY)$ defines a non-degenerate scalar product on the simple Lie algebra ${\mathfrak{sl}}(n,{\mathbb R})$. Diagonalizing the Gram matrix $Gr$ associated with this scalar product we find a basis of ${\mathfrak{sl}}(n,{\mathbb R})$ of eigenvectors of $Gr$ which produces a family of bracket generating distributions on ${\rm SL}(n,{\mathbb R})$. Consequently, the bilinear form $B$ defines sub-pseudo-Riemannian structures on these distributions. Each of these geometric structures naturally carries a metric quadratic Hamiltonian. In the present paper, we construct in detail these manifolds, study Poisson-commutation relations between different Hamiltonians, and present some explicit solutions of the corresponding Hamiltonian system for $n=2$.
\end{abstract}

\section{Introduction}

For several decades, sub-Riemannian geometry has been a source of very interesting mathematics combining differential geometry, geometric control theory, mathematical physics, geometric analysis, and Lie theory, see \cite{ABB,M}. The most important example studied in this area is the three-dimensional Heisenberg group, see \cite{CDPT}, which has become the quintessential sub-Riemannian manifold. It is natural, therefore, that it was one of the first manifolds endowed with a sub-Lorentzian structure, see \cite{G1}. Afterwards, the study of non-degenerate metric structures on bracket-generating distributions has become a topic of significant interest in geometric analysis, for example see \cite{BFM, BT,GKM,G2,G3,GV,PP,SS1,SS2}. These manifolds have been called sub-pseudo-Riemannian in the literature.

In order to understand deeply the behavior of sub-Riemannian metric structures, the classical low-dimensional real Lie groups have been important examples, see \cite{BZ1,BZ2,BS,B,BR,DAC,VG}, not only because many complicated computations can be performed explicitly, but due to their ubiquity in many areas of mathematics. Reference \cite{GV} is of special interest for the present paper, where the Lie group ${\rm SU}(1,1)$ is endowed with two metric structures, a sub-Riemannian and a sub-Lorentzian one. Since ${\rm SU}(1,1)\cong{\rm SL}(2,{\mathbb R})$, some of the results in the present paper can be considered generalizations of the sub-Lorentzian geometric results obtained in the aforementioned reference.

In this article we study a family of sub-pseudo-Riemannian metrics on ${\rm SL}(n,{\mathbb R})$, $n\geq2$, arising naturally from its Killing form. The initial steps of the construction follow from elementary linear algebra, but they have to be done carefully not to make the computations unnecessarily complicated. The argument provides a Lie algebra generating set of ${\mathfrak{sl}}(n,{\mathbb R})$, that is, a set of vectors that together with Lie brackets between them generate ${\mathfrak{sl}}(n,{\mathbb R})$ as a vector space. This set, which is formed of eigenvectors of the Gram matrix of the Killing form, is endowed naturally with a non-degenerate metric, and thus it produces sub-pseudo-Riemannian structures on ${\rm SL}(n,{\mathbb R})$, defined on a family of bracket generating distributions obtained by left translating the generating vectors. For these manifolds, we study the standard metric quadratic Hamiltonians and show that some of them Poisson commute. Considering that the algebraic expressions, and thus the associated Hamiltonian system, become increasingly cumbersome in higher dimensions, we present a few conserved quantities for $n=2$ and find some explicit solutions to the system.

The paper is organized as follows. After a brief summary of preliminaries in Section \ref{sec:prelim}, we present the construction of the family of bracket generating distributions ${\mathscr D}_{\mathcal I}$ on ${\rm SL}(n,{\mathbb R})$ in Section \ref{sec:framework}. These are the distributions that define the main objects of this article. In Section \ref{sec:ham} we study the corresponding Hamiltonian functions defined by the restriction of the Killing form of ${\mathfrak{sl}}(n,{\mathbb R})$. After showing a general result concerning the Poisson commutativity of some of these Hamiltonians, we study in more depth the ${\rm SL}(2,{\mathbb R})$ case, in which the construction of the previous section simplifies significantly.

\section{Preliminaries}\label{sec:prelim}

\subsection{Concerning ${\rm SL}(n,{\mathbb R})$}\label{ssec:sl}

Let $n\geq2$ be an integer. We denote by $G={\rm SL}(n,{\mathbb R})$ the Lie group of all matrices of size $n\times n$ of determinant one. It is well-known that its Lie algebra ${\mathfrak{g}}={\mathfrak{sl}}(n,{\mathbb R})$ consists of all traceless matrices of size $n\times n$. A standard result on the theory of Lie algebras is that ${\mathfrak{g}}$ is simple, that is, contains no nonzero proper ideals and it is not abelian. In the case of real Lie algebras, simplicity is equivalent to the fact that the Killing form
\[
\kappa(X,Y)={\rm tr}({\rm ad}_X\circ{\rm ad}_Y),\quad X,Y\in{\mathfrak{g}}, 
\]
is a non-degenerate bilinear form. As usual, the adjoint representation ${\rm ad}\colon{\mathfrak{g}}\to{\rm End}({\mathfrak{g}})$ is given by ${\rm ad}_X(Z)=[X,Z]=XZ-ZX$. In the specific case of ${\mathfrak{g}}$, the Killing form $\kappa$ is simply
\[
\kappa(X,Y)=2n\cdot{\rm tr}(XY).
\]
For more details, the interested reader can consult \cite{H1,H2}.

Let ${\rm id}_n\in G$ denote the $n\times n$ identity matrix. Given a tangent vector $v\in T_{{\rm id}_n}G\cong{\mathfrak{g}}$, we obtain a left-invariant vector field $V\in{\mathfrak{X}}(G)$ by left translation, that is, we have
\[
V(x)=d_{{\rm id}_n}L_x(v),\quad x\in G,
\]
where $L_x\colon G\to G$ is the smooth function given by $L_x(h)=xh$. Since $G$ is a matrix Lie group, the previous rather abstract definition becomes $V(x)=xv$.

\subsection{Non-degenerate metrics on bracket generating distributions}\label{ssec:bg}

Given a smooth manifold $M$, a distribution $D\hookrightarrow TM$ is a vector sub-bundle of the tangent bundle $TM$. The celebrated Chow-Rashevski{\u\i} theorem, see \cite{C,R} for the original and \cite{M} for the modern formulation, states that if $M$ is connected and the sections of $D$ together with all possible iterated Lie brackets of sections of $D$ span the entire tangent space at every point, then for any two points $x_0,x_1\in M$ there exists an absolutely continuous curve $\gamma\colon[0,1]\to M$ such that
\begin{equation}\label{eq:hor}
\gamma(0)=x_0,\quad\gamma(1)=x_1\quad\mbox{and}\quad\dot\gamma(t)\in D|_{\gamma(t)}\quad\mbox{a.e.}
\end{equation}
These types of distributions are called bracket generating and a curve $\gamma$ satisfying \eqref{eq:hor} is called horizontal or admissible.

Given a distribution $D$ on a smooth manifold, we can endow $D$ with a non-degenerate metric tensor $g$, that is, a non-degenerate bilinear form at each point $q\in M$
\[
g_q\colon D_q\times D_q\to{\mathbb R}
\]
smoothly varying with respect to $q$. If the metric is positive definite, then one can define the so-called Carnot-Carath\'eodory distance $d_{CC}$ by
\begin{equation}\label{eq:CC}
d_{CC}(x_0,x_1)=\inf\left\{\int_0^1g(\dot\gamma(t),\dot\gamma(t))^{1/2}dt\colon\gamma\mbox{ satisfies \eqref{eq:hor}}\right\}.
\end{equation}
If the distribution $D$ is bracket generating and the manifold $M$ is connected, then $d_{CC}(x_0,x_1)$ is finite for any pair of points $x_0,x_1\in M$, see \cite{ABB,M}.

In the case the metric is not positive definite, equality \eqref{eq:CC} does not make sense in general. In order to determine an interesting set of curves, it is usual to generalize the Hamiltonian formulation of the minimization problem \eqref{eq:CC}. Assuming that that $g$ is positive definite and that the distribution $D$ is generated locally by $k$ linearly independent $g$-orthonormal vector fields $V_1,\dotsc,V_k$, then for any horizontal curve $\gamma$ there are piecewise smooth functions $f_1,\dotsc,f_k$ such that
\[
\dot\gamma=\sum_{i=1}^kf_iV_i\quad\mbox{and thus}\quad g(\dot\gamma,\dot\gamma)=\sum_{i=1}^kf_i^2.
\]
In these terms, the metric quadratic Hamiltonian $H\colon T^*M\to{\mathbb R}$ corresponding to the minimization problem \eqref{eq:CC} is
\[
H(q;\lambda)=\frac12\sum_{i=1}^kP_{V_i}(q;\lambda)^2,\quad\lambda\in T_q^*M,
\]
where $P_V(q;\lambda)=\lambda(V(q))$ is the so-called momentum function of $V$. In the general case, regardless of whether  vector fields $V_1,\dotsc,V_k$ are $g$-orthonormal, but always assuming they form a basis, then the corresponding Hamiltonian is
\begin{equation}\label{eq:Hamnonorth}
H(q;\lambda)=\frac12\sum_{i=1}^kg^{ij}(q)P_{V_i}(q;\lambda)P_{V_j}(q;\lambda),\quad\lambda\in T_q^*M,
\end{equation}
where $(g^{ij}(q))$ denotes the inverse of the metric at $q$. The solutions to the Hamiltonian system associated to $H$ are called normal geodesics in sub-Riemannian geometry, see \cite{M}.

The general definition \eqref{eq:Hamnonorth} makes sense even in the non-degenerate non-positive definite case, which is the situation that is most interesting for the present paper. In this context, it is usual to say that the metric $g$ has an index $r>0$, which is the largest dimension of a negative definite subspace of $D$.


\section{General framework}\label{sec:framework}

Consider the symmetric bilinear form $B=\frac1{2n}\kappa\colon{\mathfrak{g}}\times{\mathfrak{g}}\to{\mathbb R}$. As mentioned in Subsection \ref{ssec:sl}, the bilinear form $B$ is non-degenerate, since it is a (non-zero) scalar multiple of the Killing form. First we need to compute the Gram matrix of $B$ using a suitable ordered basis of ${\mathfrak{g}}$. Define the following elements of ${\mathfrak{g}}$:
\begin{eqnarray}\label{eq:mat}
&e_1^1=E_{1,2},\quad e_1^2=E_{2,1}\quad\dotsc\quad e_{n-1}^1=E_{1,n},\quad e_{n-1}^2=E_{n,1}&\nonumber\\
&e_n^1=E_{2,3},\quad e_n^2= E_{3,2}\quad\dotsc\quad e_{2n-3}^1=E_{2,n},\quad e_{2n-3}^2=E_{n,2}&\nonumber\\
&\vdots&\\
&e_{\frac{n(n-1)}2}^1=E_{n-1,n},\quad e_{\frac{n(n-1)}2}^2=E_{n,n-1},&\nonumber\\
&e_1=E_{n,n}-E_{1,1}\quad \dotsc\quad e_{n-1}=E_{n,n}-E_{n-1,n-1}.&\nonumber
\end{eqnarray}
where, as usual, the $n\times n$ matrix $E_{i,j}$ has a 1 in the position $(i,j)$ and the rest of entries is zero. It is an elementary exercise to check that the set
\begin{equation}\label{eq:basis}
{\mathscr B}=\left\{e_1^1,e_1^2,\dotsc,e_{\frac{n(n-1)}2}^1,e_{\frac{n(n-1)}2}^2\right\}\cup\{e_1,\dotsc,e_{n-1}\}
\end{equation}
is a basis for ${\mathfrak{g}}$. Moreover, we have that $(e_r^1)^\top=e_r^2$ for $r=1,\dotsc,\frac{n(n-1)}2$ and $(e_\ell)^\top=e_\ell$ for $\ell=1,\dotsc,n-1$.

Recall that the Gram matrix of a bilinear form $T\colon W\times W\to{\mathbb R}$ on a $d$-dimensional real vector space $W$ with respect to an ordered basis $\{w_1,\dotsc,w_d\}$ is the $d\times d$ matrix $(T(w_i,w_j))_{1\leq i,j\leq d}$. Considering the ordered basis \eqref{eq:basis}, we can provide an explicit description of the Gram matrix of $B$. This description will play an important role for finding the subspaces of ${\mathfrak{g}}$ that will be of interest in this article. 

For convenience, given a positive integer $m\in{\mathbb N}$, let us denote by ${\mathcal P}_m=(p_{i,j})$ the $m\times m$ matrix given by
\begin{equation}\label{eq:Pm}
p_{i,j}=\begin{cases}2,&\mbox{if }i=j,\\1,&\mbox{if }i\neq j.\end{cases}
\end{equation}
Matrices of this form will appear several times in what follows. It is a simple exercise to invert matrix \eqref{eq:Pm} to obtain ${\mathcal P}_m^{-1}=(p^{i,j})$, where
\[
p^{i,j}=\begin{cases}\frac{m}{m+1},&\mbox{if }i=j,\\-\frac1{m+1},&\mbox{if }i\neq j.\end{cases}
\]

With all this notation at hand, we can describe precisely the Gram matrix of $B$.

\begin{Pro}
The Gram matrix $Gr$ of the bilinear form $B$ in the ordered basis ${\mathscr{B}}$ is the $(n^2-1)\times(n^2-1)$ block diagonal matrix
\[
Gr=\begin{pmatrix}J&0&0&\cdots&0&0\\0&J&0&\cdots&0&0\\0&0&J&\cdots&0&0\\\vdots&\vdots&\vdots&\ddots&\vdots&\vdots\\0&0&0&\cdots&J&0\\0&0&0&\cdots&0&{\mathcal P}_{n-1}\end{pmatrix}
\]
in which the $2\times 2$ block $J=\begin{pmatrix}0&1\\1&0\end{pmatrix}$ appears $\dfrac{n(n-1)}2$ times.
\end{Pro}

\begin{proof}
This follows from direct computations. If $r,s=1,\dotsc,\frac{n(n-1)}2$ and $\ell=1,\dotsc,n-1$, then it is easy to see that
\[
B(e_r^1,e_s^1)=B(e_r^2,e_s^2)=B(e_r^1,e_\ell)=B(e_r^2,e_\ell)=0,\quad B(e_r^1,e_s^2)=\begin{cases}1,&\mbox{if }r=s,\\0,&\mbox{else.}\end{cases}
\]
This explains the $\frac{n(n-1)}2$ blocks $J$ in $Gr$ and all the off-block-diagonal zeros. Concerning the $(n-1)\times(n-1)$ block $P_{n-1}$, it can be determined directly from the fact that
\[
e_k e_\ell=(E_{k,k}-E_{n,n})(E_{\ell,\ell}-E_{n,n})=\begin{cases}E_{k,k}+E_{n,n}&\mbox{if }k=\ell,\\E_{n,n}&\mbox{else,}\end{cases}
\]
for $k,\ell=1,\dotsc,n-1$. As a consequence
\begin{equation*}
B(e_k,e_\ell)=\begin{cases}{\rm tr}(E_{k,k}+E_{n,n})=2&\mbox{if }k=\ell,\\{\rm tr}(E_{n,n})=1&\mbox{else.}\end{cases}\qedhere
\end{equation*}
\end{proof}

\begin{Cor}\label{cor:vecprop}
The spectrum of $Gr$ is given by
\[
\sigma(Gr)=\{n,1,-1\},
\]
where the respective eigenspaces $E_\lambda$, $\lambda\in\sigma(Gr)$, have dimensions
\[
\dim E_{n}=1,\quad\dim E_{1}=\frac{n(n-1)}2+n-2,\quad\dim E_{-1}=\frac{n(n-1)}2.
\]
\end{Cor}

\begin{proof}
It is enough to provide a basis for each eigenspace. 
\begin{description}
\item[Basis for $E_{n}$] The matrix 
\[
H=e_1+\cdots+e_{n-1}=(n-1)E_{n,n}-E_{1,1}-\cdots-E_{n-1,n-1}
\]
corresponds to an eigenvector of $Gr$ with eigenvalue $n$.

\item[Basis for $E_{1}$] The linearly independent matrices
\[
X_1=e_1^1+e_1^2,\dotsc,X_{\frac{n(n-1)}2}=e_{\frac{n(n-1)}2}^1+e_{\frac{n(n-1)}2}^2,
\]
\[
Z_1=e_1-e_{n-1},\dotsc,Z_{n-2}=e_{n-2}-e_{n-1}
\]
correspond to $\frac{n(n-1)}2+n-2$ eigenvectors of $Gr$ with eigenvalue 1. Evidently, the vectors $Z_\ell$ only appear when $n\geq3$.

\item[Basis for $E_{-1}$] The linearly independent matrices
\[
Y_1=e_1^1-e_1^2,\dotsc,Y_{\frac{n(n-1)}2}=e_{\frac{n(n-1)}2}^1-e_{\frac{n(n-1)}2}^2
\]
correspond to $\frac{n(n-1)}2$ eigenvectors of $Gr$ with eigenvalue $-1$.\qedhere
\end{description}
\end{proof}

The goal now is to show that the smallest Lie subalgebra of ${\mathfrak{g}}$ containing the linearly independent set 
\[
{\mathscr{C}}=\left\{X_1,\dotsc,X_{\frac{n(n-1)}2},Y_1,\dotsc,Y_{\frac{n(n-1)}2}\right\}
\]
is the full algebra ${\mathfrak{g}}$. We say that ${\mathscr{C}}$ is a Lie algebra generating set for ${\mathfrak{g}}$. To show this, we need the following elementary formula, see \cite[p. 2]{H2}.

\begin{Le}\label{lem:comm}
Let $i,j,k,\ell\in\{1,\dotsc,n\}$, then
\[
[E_{i,j},E_{k,\ell}]=E_{i,j}E_{k,\ell}-E_{k,\ell}E_{i,j}=\begin{cases}E_{i,\ell}&\mbox{if }i\neq\ell,j=k,\\-E_{k,j}&\mbox{if }i=\ell,j\neq k,\\E_{i,i}-E_{j,j}&\mbox{if }i=\ell,j=k,\\0&\mbox{else.}\end{cases}
\]
\end{Le}

As a consequence of this formula, we can state the following algebraic fact.

\begin{Th}\label{th:bg}
The collection ${\mathscr{C}}$ of traceless matrices is a Lie algebra generating set for ${\mathfrak{g}}$.
\end{Th}

\begin{proof}
Since the subspace
\[
{\mathscr{D}}_{XY}={\rm span}\left\{X_1,\dotsc,X_{\frac{n(n-1)}2},Y_1,\dotsc,Y_{\frac{n(n-1)}2}\right\}
\]
of ${\mathfrak{g}}$ has codimension $n-1$, it is enough to show that $Z_1,\dotsc,Z_{n-2}$ and $H$ can be written in terms of Lie brackets of elements of ${\mathscr{D}}_{XY}$.

First notice that if $X_r=E_{p,q}+E_{q,p}$ and $Y_r=E_{p,q}-E_{q,p}$ according to the renumbering \eqref{eq:mat}, where $1\leq p<q\leq n$, then following Lemma \ref{lem:comm} we have
\[
[X_r,Y_r]=[E_{p,q}+E_{q,p},E_{p,q}-E_{q,p}]=[E_{q,p},E_{p,q}]-[E_{p,q},-E_{q,p}]=2(E_{q,q}-E_{p,p}).
\]
The vector $H$ can be easily obtained as
\begin{equation*}
\sum_{i=1}^{n-1}[E_{i,n}+E_{n,i},E_{i,n}-E_{n,i}]=2\sum_{i=1}^{n-1}(E_{n,n}-E_{i,i})=2H\in[{\mathscr{D}}_{XY},{\mathscr{D}}_{XY}].
\end{equation*}
If $n=2$, we are done. In the case that $n\geq3$, then for $1\leq\ell\leq n-2$, we see that 
\[
[E_{\ell,n-1}+E_{n-1,\ell},E_{\ell,n-1}-E_{n-1,\ell}]=2(E_{n-1,n-1}-E_{\ell,\ell})=2Z_\ell\in[{\mathscr{D}}_{XY},{\mathscr{D}}_{XY}].\qedhere
\]
\end{proof}

As an immediate consequence of the previous result, we have that for an arbitrary subset ${\mathcal I}$ of $\{1,\dotsc,n-2\}$, $n\geq3$, the collection
\begin{equation}\label{eq:basis}
{\mathscr C}_{\mathcal I}=\left\{X_1,\dotsc,X_{\frac{n(n-1)}2},Y_1,\dotsc,Y_{\frac{n(n-1)}2}\right\}\cup\{Z_\ell\colon\ell\in{\mathcal I}\}.
\end{equation}
is also a Lie algebra generating set for ${\mathfrak{g}}$. As it turns out, this collection of vectors is not orthogonal, but the metric restricted to 
\begin{equation}\label{eq:dist}
{\mathscr{D}}_{\mathcal I}={\rm span}\,{\mathscr C}_{\mathcal I}
\end{equation}
can be easily described as follows.

\begin{Pro}\label{prop:norm}
The Gram matrix for the bilinear form $B$ restricted to ${\mathscr{D}}_{\mathcal I}$ in the ordered basis ${\mathscr{C}}_{\mathcal I}$ is the $(n^2-n+\#{\mathcal I})\times(n^2-n+\#{\mathcal I})$ symmetric block matrix given by
\[
g=\begin{pmatrix}2\,{\rm id}_{\frac{n(n-1)}2}&0&0\\0&-2\,{\rm id}_{\frac{n(n-1)}2}&0\\0&0&{\mathcal P}_{\#{\mathcal I}}\end{pmatrix}.
\]
\end{Pro}

\begin{proof}
For $i,j\in\left\{1,\dotsc,\frac{n(n-1)}2\right\}$ and for $k,\ell\in\{1,\dotsc,n-2\}$, the following equalities can be computed directly from the expressions found in Corollary \ref{cor:vecprop}
\[
B(X_i,Y_j)=B(X_i,Z_k)=B(Y_i,Z_k)=0,\quad B(X_i,X_j)=-B(Y_i,Y_j)=\begin{cases}2&\mbox{if }i=j,\\0&\mbox{else,}\end{cases}
\]
and
\[
B(Z_k,Z_\ell)=\begin{cases}2&\mbox{if }k=\ell,\\1&\mbox{else.}\end{cases}\qedhere
\]
\end{proof}

\begin{Ej}
For $n=2$ and $n=3$, this construction can be easily described. For instance, in ${\mathfrak{sl}}(2,{\mathbb R})$, the matrices presented in \eqref{eq:mat} are
\[
e_1^1=\begin{pmatrix}0&1\\0&0\end{pmatrix},\quad e_1^2=\begin{pmatrix}0&0\\1&0\end{pmatrix},\quad e_1=\begin{pmatrix}-1&0\\0&1\end{pmatrix}.
\]
It is easy to check that
\[
{\rm tr}(e_1^1e_1^1)={\rm tr}(e_1^2e_1^2)={\rm tr}(e_1^1e_1)={\rm tr}(e_1^2e_1)=0,\quad{\rm tr}(e_1^1e_1^2)=1,\quad{\rm tr}(e_1e_1)=2.
\]
Then the Gram matrix is simply $Gr=\begin{pmatrix}0&1&0\\1&0&0\\0&0&2\end{pmatrix}$ and the basis of eigenvectors of $B$ is
\[
H=\begin{pmatrix}-1&0\\0&1\end{pmatrix},\quad X_1=\begin{pmatrix}0&1\\1&0\end{pmatrix},\quad Y_1=\begin{pmatrix}0&1\\-1&0\end{pmatrix}.
\]
Concretely, since every matrix $L=\begin{pmatrix}-a&b\\c&a\end{pmatrix}=ae_1+be_1^1+ce_1^2\in{\mathfrak{sl}}(2,{\mathbb R})$ can be written uniquely as
\[
L=aH+\frac{b+c}2X_1+\frac{b-c}2Y_1
\]
then we obtain
\begin{align*}
B(L,L)&=a^2B(H,H)+\left(\frac{b+c}2\right)^2B(X_1,X_1)+\left(\frac{b-c}2\right)^2B(Y_1,Y_1)\\
&=2a^2+2\left(\frac{b+c}2\right)^2-2\left(\frac{b-c}2\right)^2=2a^2+2bc.
\end{align*}
Similarly, for ${\mathfrak{sl}}(3,{\mathbb R})$, the matrices presented in \eqref{eq:mat} are
\[
e_1^1=\begin{pmatrix}0&1&0\\0&0&0\\0&0&0\end{pmatrix},\quad e_1^2=\begin{pmatrix}0&0&0\\1&0&0\\0&0&0\end{pmatrix},\quad e_2^1=\begin{pmatrix}0&0&1\\0&0&0\\0&0&0\end{pmatrix},\quad e_2^2=\begin{pmatrix}0&0&0\\0&0&0\\1&0&0\end{pmatrix},
\]
\[
e_3^1=\begin{pmatrix}0&0&0\\0&0&1\\0&0&0\end{pmatrix},\quad e_3^2=\begin{pmatrix}0&0&0\\0&0&0\\0&1&0\end{pmatrix},\quad e_1=\begin{pmatrix}-1&0&0\\0&0&0\\0&0&1\end{pmatrix},\quad e_2=\begin{pmatrix}0&0&0\\0&-1&0\\0&0&1\end{pmatrix}.
\]
The Gram matrix is the $8\times8$ block diagonal matrix
\[
Gr=\begin{pmatrix}J&0&0&0\\0&J&0&0\\0&0&J&0\\0&0&0&{\mathcal P}_2\end{pmatrix},
\]
and the basis of eigenvectors is
\[
X_1=e_1^1=\begin{pmatrix}0&1&0\\1&0&0\\0&0&0\end{pmatrix},\quad X_2=\begin{pmatrix}0&0&1\\0&0&0\\1&0&0\end{pmatrix},\quad X_3=\begin{pmatrix}0&0&0\\0&0&1\\0&1&0\end{pmatrix},
\]
\[
Y_1=e_1^1=\begin{pmatrix}0&1&0\\-1&0&0\\0&0&0\end{pmatrix},\quad Y_2=\begin{pmatrix}0&0&1\\0&0&0\\-1&0&0\end{pmatrix},\quad Y_3=\begin{pmatrix}0&0&0\\0&0&1\\0&-1&0\end{pmatrix},
\]
\[
Z_1=\begin{pmatrix}-1&0&0\\0&1&0\\0&0&0\end{pmatrix},\quad H=\begin{pmatrix}-1&0&0\\0&-1&0\\0&0&2\end{pmatrix}.
\]
A similar --but lengthier-- computation as before, shows that for an arbitrary traceless $3\times3$ matrix 
\begin{align*}
L=\begin{pmatrix}-a&b&c\\d&-e&f\\g&h&a+e\end{pmatrix}&=\frac{a-e}2Z_1+\frac{a+e}2H+\frac{b+d}2X_1+\frac{b-d}2Y_1\\
&+\frac{c+g}2X_2+\frac{c-g}2Y_2+\frac{f+h}2X_3+\frac{f-h}2Y_3
\end{align*}
in ${\mathfrak{sl}}(3,{\mathbb R})$, we have
\[
B(L,L)=2(a^2+a e+e^2+b d+cg+fh).
\]
\end{Ej}

As a consequence of Theorem \ref{th:bg}, by left translating ${\mathscr{D}}_{\mathcal I}$ from equation \eqref{eq:dist} to the entire group $G$, we obtain the following geometric result 

\begin{Th}
For an arbitrary subset ${\mathcal I}$ of $\{1,\dotsc,n-2\}$, $n\geq3$, the left-translation of the subspace ${\mathscr{D}}_{\mathcal I}\subset{\mathfrak{g}}$
defines a distribution on $G={\rm SL}(n,{\mathbb R})$ on which the bilinear form $B$ defines a bi-invariant sub-pseudo-Riemannian metric. The corank of ${\mathscr{D}}_{\mathcal I}$ is $n-1-\#{\mathcal I}$ and the index of $B$ restricted to ${\mathscr{D}}_{\mathcal I}$ is ${n(n-1)}/2$.
\end{Th}

The case of $n=2$ is well-known in the literature and both its sub-Riemannian and sub-Lorentzian geometries have been explored in depth in \cite{GV} using the fact that ${\rm SL}(2,{\mathbb R})\cong{\rm SU}(1,1)$ and that the projection ${\rm SU}(1,1)\to{\rm SU}(1,1)/{\rm U}(1)$ is a principal ${\rm U}(1)$-bundle. The main idea in the aforementioned reference is to horizontally lift appropriate curves in the two dimensional K\"ahler manifold ${\rm SU}(1,1)/{\rm U}(1)$ using the ${\rm U}(1)$-action. This approach does not seem generalize easily to the higher dimensional situation.

In the following section we will study the Hamiltonian system on $T^*G$ associated to the metric quadratic Hamiltonian function corresponding to the pseudo-Riemannian metric defined on the family of bracket-generating distributions ${\mathscr{D}}_{\mathcal I}$.

\section{Hamiltonian formalism}\label{sec:ham}

It is a well-known strategy in sub-Riemannian geometry that normal geodesics can be found using the classical Hamiltonian formalism related to the metric, as briefly mentioned in Subsection \ref{ssec:bg}. Even though in sub-pseudo-Riemannian geometry these curves do not possess the same locally distance minimizing properties than in the positive definite case, we can still describe geometric properties of the flow of the corresponding metric quadratic Hamiltonian functions ${\mathscr{H}}_{\mathcal I}\colon T^*G\to{\mathbb R}$, for ${\mathcal I}\subset\{1,\dotsc,n-2\}$, $n\geq3$. The case $n=2$ will be dealt with separately in Subsection \ref{ssec:n=2}. 

\subsection{Coordinate description of the metric Hamiltonian}

In what follows we will use fact that $G$ is a hypersurface in $n^2$-dimensional Euclidean space
\begin{equation}\label{eq:coord}
{\mathbb R}^{n\times n}=\{A=(a_{i,j})\colon 1\leq i,j\leq n\}.
\end{equation}

The idea is to describe explicitly the Hamiltonian function 
\begin{equation}\label{eq:Ham}
{\mathscr{H}}_{\mathcal I}=\frac14\sum_{r=1}^{\frac{n(n-1)}2}(P_{{\mathcal X}_r}^2-P_{{\mathcal Y}_r}^2)+\frac12\sum_{k,\ell\in{\mathcal I}}p^{k,\ell}P_{{\mathcal Z}_k}P_{{\mathcal Z}_\ell},
\end{equation}
where ${\mathcal X}_r$, ${\mathcal Y}_r$ and ${\mathcal Z}_\ell$ represent the the left-invariant vector fields on $G$ associated to the vectors $X_r$, $Y_r$ and $Z_\ell$ described in Corollary \ref{cor:vecprop}, respectively, and $P_V$ was described in Subsection  \ref{ssec:bg}. The difference in the first factor of the Hamiltonian \eqref{eq:Ham} compared to \eqref{eq:Hamnonorth} is to conform to the norms found in Proposition \ref{prop:norm}. 

Recall that, since $G$ is a matrix Lie group, the left translation of $v\in{\mathfrak{g}}$ to the tangent space $T_AG$, for $A\in G$, is given by
\[
d_{{\rm id}_n}L_A(v)=Av.
\]
Therefore, the vector fields ${\mathcal X}_r$, ${\mathcal Y}_r$ and ${\mathcal Z}_\ell$ are given by
\[
{\mathcal X}_r|_A=AX_r,\quad{\mathcal Y}_r|_A=AY_r,\quad{\mathcal Z}_\ell|_A=AZ_\ell,
\]
for $A\in G$. To be more explicit, for $A\in G$ and $1\leq p<q\leq n$, define the $n\times n$ matrices $A^+_{p,q}$ and $A^-_{p,q}$ by
\begin{equation}\label{eq:XY}
A(E_{p,q}\pm E_{q,p})=A^\pm_{p,q},
\end{equation}
that is, $A^+_{p,q}$ is the matrix with all entries zero except that its $p^{\rm th}$ column is the $q^{\rm th}$ column of $A$ and its $q^{\rm th}$ column is the $p^{\rm th}$ column of $A$, and similarly $A^-_{p,q}$ is the matrix with all entries zero except that its $p^{\rm th}$ column is the $q^{\rm th}$ column of $A$ and its $q^{\rm th}$ column is minus the $p^{\rm th}$ column of $A$. This describes completely the vector fields ${\mathcal X}_r$ and ${\mathcal Y}_r$ according to the renumbering \eqref{eq:mat}. Analogously, for $n\geq3$ and $1\leq\ell\leq n-2$, define the $n\times n$ matrices $A_\ell$
\begin{equation}\label{eq:Z}
A(E_{\ell,\ell}-E_{n-1,n-1})=A_\ell
\end{equation}
is the matrix with all entries zero except that its $\ell^{\rm th}$ column is the $\ell^{\rm th}$ column of $A$ and its $(n-1)^{\rm st}$ column is minus the $(n-1)^{\rm st}$ column of $A$. This describes completely the vector fields ${\mathcal Z}_\ell$.

Having described the nature of all vector fields that generate the distributions ${\mathscr{D}}_{\mathcal I}$, we can present explicitly the Hamiltonian ${\mathscr{H}}_{\mathcal I}$ in coordinates.

\begin{Th}\label{th:Hamcoord}
For a given subset ${\mathcal I}\subset\{1,\dotsc,n-2\}$ the metric quadratic Hamiltonian ${\mathscr{H}}_{\mathcal I}$ is given by
\begin{multline}\label{eq:Hamcoord}
{\mathscr{H}}_{\mathcal I}(a;\lambda)=\frac12\sum_{{\substack{p,q=1\\p\neq q}}}^n\sum_{i,j=1}^na_{i,p}a_{j,q}\lambda_{j,p}\lambda_{i,q}\\
+\frac12\sum_{k,\ell\in{\mathcal I}}\sum_{i,j=1}^np^{k,\ell}(a_{i,k}\lambda_{i,k}-a_{i,n-1}\lambda_{i,n-1})(a_{j,\ell}\lambda_{j,\ell}-a_{j,n-1}\lambda_{j,n-1})
\end{multline}
for $(a;\lambda)\in T^*G$ in the Euclidean coordinates induced by \eqref{eq:coord}.
\end{Th}

\begin{proof}
In coordinates, the vector fields ${\mathcal X}_r={\mathcal X}_{p,q}$, ${\mathcal Y}_r={\mathcal Y}_{p,q}$ and ${\mathcal Z}_\ell$ are given by
\[
{\mathcal X}_{p,q}=\sum_{i=1}^n\left(a_{i,q}\frac{\partial}{\partial a_{i,p}}+a_{i,p}\frac{\partial}{\partial a_{i,q}}\right),\quad{\mathcal Y}_{p,q}=\sum_{i=1}^n\left(a_{i,q}\frac{\partial}{\partial a_{i,p}}-a_{i,p}\frac{\partial}{\partial a_{i,q}}\right),
\]
\[
{\mathcal Z}_\ell=\sum_{i=1}^n\left(a_{i,\ell}\frac{\partial}{\partial a_{i,\ell}}-a_{i,n-1}\frac{\partial}{\partial a_{i,n-1}}\right),
\]
for $1\leq p<q\leq n$ and $\ell\in\{1,\dotsc,n-2\}$. These follow directly from the descriptions after equalities \eqref{eq:XY} and \eqref{eq:Z} respectively.

Since we are considering $G\hookrightarrow{\mathbb R}^{n\times n}$, we define as usual the covector $\lambda_{i,j}$ as dual to the coordinate vector field $\dfrac{\partial}{\partial a_{i,j}}$. Therefore, the corresponding momentum functions are simply
\[
P_{{\mathcal X}_{p,q}}=\sum_{i=1}^n\left(a_{i,q}\lambda_{i,p}+a_{i,p}\lambda_{i,q}\right),\quad P_{{\mathcal Y}_{p,q}}=\sum_{i=1}^n\left(a_{i,q}\lambda_{i,p}-a_{i,p}\lambda_{i,q}\right),
\]
\[
P_{{\mathcal Z}_\ell}=\sum_{i=1}^n\left(a_{i,\ell}\lambda_{i,\ell}-a_{i,n-1}\lambda_{i,n-1}\right).
\]

Substituting these expressions for the momenta into equation \eqref{eq:Ham}, we have
\begin{align*}
{\mathscr{H}}_{\mathcal I}(a;\lambda)&=\frac14\sum_{1\leq p<q\leq n}\left(\left(\sum_{i=1}^n\left(a_{i,q}\lambda_{i,p}+a_{i,p}\lambda_{i,q}\right)\right)^2-\left(\sum_{i=1}^n\left(a_{i,q}\lambda_{i,p}-a_{i,p}\lambda_{i,q}\right)\right)^2\right)\\
&+\frac12\sum_{k,\ell\in{\mathcal I}}\sum_{i,j=1}^np^{k,\ell}(a_{i,k}\lambda_{i,k}-a_{i,n-1}\lambda_{i,n-1})(a_{j,\ell}\lambda_{j,\ell}-a_{j,n-1}\lambda_{j,n-1}).
\end{align*}

Focusing only on the term ${\mathscr{H}}_{\varnothing}$, we have
\[
{\mathscr{H}}_{\varnothing}(a;\lambda)=\sum_{1\leq p<q\leq n}\left(\sum_{j=1}^na_{j,q}\lambda_{j,p}\right)\left(\sum_{i=1}^na_{i,p}\lambda_{i,q}\right)=\sum_{1\leq p<q\leq n}\left(\sum_{i,j=1}^na_{j,q}\lambda_{j,p}a_{i,p}\lambda_{i,q}\right),
\]
from which equality \eqref{eq:Hamcoord} follows directly, taking into account that in the first sum we simply double the index set and change variables.
\end{proof}

Following the usual techniques from Hamiltonian systems, we can write down immediately the system of ODEs on $T^*G$ that correspond to ${\mathscr H}_{\mathcal I}$ and starting at $a^0=(a_{i,j}^0)$ and $\lambda^0=(\lambda_{i,j}^0)$ at time $t=0$ 

\begin{Cor}
The integral curves of the metric quadratic Hamiltonian \eqref{eq:Hamcoord} starting from $(a_{i,j}^0)\in{\mathfrak{g}}$ with initial covector $\lambda=(\lambda_{i,j}^0)\in{\mathfrak g}^*$ satisfy the system
\begin{equation}\label{eq:Hamsys}
\begin{cases}\dot a_{i,j}=\dfrac{\partial{\mathscr H}_{\mathcal I}}{\partial\lambda_{i,j}}\\
\dot\lambda_{i,j}=-\dfrac{\partial{\mathscr H}_{\mathcal I}}{\partial a_{i,j}}\\
a_{i,j}(0)=a_{i,j}^0\\\lambda_{i,j}(0)=\lambda_{i,j}^0\end{cases}
\end{equation}
\end{Cor}

Solving the Hamiltonian system \eqref{eq:Hamsys} for the general case seems to be unmanageable. Nevertheless, as the following subsection shows, we can, in principle, reduce the computations since the Hamiltonian function corresponding to the vector fields ${\mathcal Z}_\ell$ commutes with the one associated to the vector fields ${\mathcal X}_{p,q}$ and ${\mathcal Y}_{p,q}$.

\subsection{Commuting Hamiltonians}

For any nonempty subset ${\mathcal I}\subset\{1,\dotsc,n-2\}$, $n\geq3$, denote by
\[
{\mathscr H}_{\mathcal I}^Z={\mathscr H}_{\mathcal I}-{\mathscr H}_{\varnothing}
\]
the quadratic Hamiltonian associated only to the vector fields ${\mathcal Z}_\ell$, $\ell\in{\mathcal I}$. Then we have

\begin{Th}
The Hamiltonians ${\mathscr H}_{\varnothing}$ and ${\mathscr H}_{\mathcal I}^Z$ Poisson-commute for any nonempty subset ${\mathcal I}\subset\{1,\dotsc,n-2\}$, $n\geq3$, that is
\begin{equation}\label{eq:Hamcomm}
\{{\mathscr H}_{\varnothing},{\mathscr H}_{\mathcal I}^Z\}=\sum_{s,t=1}^n\left(\frac{\partial{\mathscr H}_{\varnothing}}{\partial a_{s,t}}\frac{\partial{\mathscr H}_{\mathcal I}^Z}{\partial\lambda_{s,t}}-\frac{\partial{\mathscr H}_{\varnothing}}{\partial\lambda_{s,t}}\frac{\partial{\mathscr H}_{\mathcal I}^Z}{\partial a_{s,t}}\right)=0.
\end{equation}
\end{Th}

\begin{proof}
Due to the coordinate form of the Hamiltonian presented in \eqref{eq:Hamcoord}, we see immediately that
\begin{align}
\frac{\partial{\mathscr H}_{\varnothing}}{\partial a_{s,t}}&
=\frac12\sum_{{\substack{p,q=1\\p\neq q}}}^n\sum_{i,j=1}^n\frac{\partial}{\partial a_{s,t}}(a_{i,p}a_{j,q})\lambda_{j,p}\lambda_{i,q}
=\frac12\sum_{{\substack{p,q=1\\p\neq q}}}^n\sum_{i=1}^n(a_{i,q}\lambda_{s,q}+a_{i,p}\lambda_{s,p})\lambda_{i,t},\label{eq:DHemptyDa}\\
\frac{\partial{\mathscr H}_{\varnothing}}{\partial\lambda_{s,t}}&
=\frac12\sum_{{\substack{p,q=1\\p\neq q}}}^n\sum_{i,j=1}^na_{i,p}a_{j,q}\frac{\partial}{\partial\lambda_{s,t}}(\lambda_{j,p}\lambda_{i,q})
=\frac12\sum_{{\substack{p,q=1\\p\neq q}}}^n\sum_{i=1}^na_{i,t}(a_{s,q}\lambda_{i,q}+a_{s,p}\lambda_{i,p}),\label{eq:DHemptyDalpha}\\
\frac{\partial{\mathscr H}_{\mathcal I}^Z}{\partial a_{s,t}}&
=\begin{cases}\displaystyle{\lambda_{s,t}\sum_{\ell\in{\mathcal I}}\sum_{i=1}^np^{t,\ell}(a_{i,\ell}\lambda_{i,\ell}-a_{i,n-1}\lambda_{i,n-1})},&\mbox{if }t\in{\mathcal I},\\
\displaystyle{\frac{\lambda_{s,n-1}}2\sum_{k,\ell\in{\mathcal I}}\sum_{i=1}^np^{k,\ell}(2a_{i,n-1}\lambda_{i,n-1}-a_{i,k}\lambda_{i,k}-a_{i,\ell}\lambda_{i,\ell})},&\mbox{if }t=n-1,\\
0,&\mbox{if }t\notin{\mathcal I}\cup\{n-1\},\end{cases}\label{eq:DHZDa}\\
\frac{\partial{\mathscr H}_{\mathcal I}^Z}{\partial\lambda_{s,t}}&=\begin{cases}\displaystyle{a_{s,t}\sum_{\ell\in{\mathcal I}}\sum_{i=1}^np^{t,\ell}(a_{i,\ell}\lambda_{i,\ell}-a_{i,n-1}\lambda_{i,n-1})},&\mbox{if }t\in{\mathcal I},\\
\displaystyle{\frac{a_{s,n-1}}2\sum_{k,\ell\in{\mathcal I}}\sum_{i=1}^np^{k,\ell}(2a_{i,n-1}\lambda_{i,n-1}-a_{i,k}\lambda_{i,k}-a_{i,\ell}\lambda_{i,\ell})},&\mbox{if }t=n-1,\\
0,&\mbox{if }t\notin{\mathcal I}\cup\{n-1\}\label{eq:DHZDalpha}.\end{cases}
\end{align}


To shorten some computations later on, we denote 
\[
F_t(a;\lambda)=\frac12\sum_{\ell\in{\mathcal I}}\sum_{j=1}^np^{t,\ell}(a_{j,\ell}\lambda_{j,\ell}-a_{j,n-1}\lambda_{j,n-1}),
\]
for $t\in{\mathcal I}$ and
\[
F(a;\lambda)=\frac14\sum_{k,\ell\in{\mathcal I}}\sum_{j=1}^np^{k,\ell}(2a_{j,n-1}\lambda_{j,n-1}-a_{j,k}\lambda_{j,k}-a_{j,\ell}\lambda_{j,\ell}),
\]
for the common factors in $\dfrac{\partial{\mathscr H}_{\mathcal I}^Z}{\partial a_{s,t}}$ and $\dfrac{\partial{\mathscr H}_{\mathcal I}^Z}{\partial\lambda_{s,t}}$, in the two corresponding non-trivial cases.

Now to compute the Poisson bracket \eqref{eq:Hamcomm} using the formulas \eqref{eq:DHemptyDa}--\eqref{eq:DHZDalpha}, we separate it into two terms:
\[
\{{\mathscr H}_{\varnothing},{\mathscr H}_{\mathcal I}^Z\}=\underbrace{\sum_{s=1}^n\sum_{t\in{\mathcal I}}\left(\frac{\partial{\mathscr H}_{\varnothing}}{\partial a_{s,t}}\frac{\partial{\mathscr H}_{\mathcal I}^Z}{\partial\lambda_{s,t}}-\frac{\partial{\mathscr H}_{\varnothing}}{\partial\lambda_{s,t}}\frac{\partial{\mathscr H}_{\mathcal I}^Z}{\partial a_{s,t}}\right)}_{\star}+
\underbrace{\sum_{s=1}^n\left(\frac{\partial{\mathscr H}_{\varnothing}}{\partial a_{s,n-1}}\frac{\partial{\mathscr H}_{\mathcal I}^Z}{\partial\lambda_{s,n-1}}-\frac{\partial{\mathscr H}_{\varnothing}}{\partial\lambda_{s,n-1}}\frac{\partial{\mathscr H}_{\mathcal I}^Z}{\partial a_{s,n-1}}\right)}_{\diamond}.
\]
For the term $\star$, we obtain
\begin{align*}
\sum_{s=1}^n\sum_{t\in{\mathcal I}}\left(\frac{\partial{\mathscr H}_{\varnothing}}{\partial a_{s,t}}\frac{\partial{\mathscr H}_{\mathcal I}^Z}{\partial\lambda_{s,t}}-\frac{\partial{\mathscr H}_{\varnothing}}{\partial\lambda_{s,t}}\frac{\partial{\mathscr H}_{\mathcal I}^Z}{\partial a_{s,t}}\right)&
=\sum_{t\in{\mathcal I}}F_t(a;\lambda)\left(\sum_{s=1}^n\sum_{{\substack{p,q=1\\p\neq q}}}^n\sum_{i=1}^na_{s,t}\lambda_{i,t}(a_{i,q}\lambda_{s,q}+a_{i,p}\lambda_{s,p})\right.\\
&\left.-\sum_{s=1}^n\sum_{{\substack{p,q=1\\p\neq q}}}^n\sum_{i=1}^na_{i,t}\lambda_{s,t}(a_{s,q}\lambda_{i,q}+a_{s,p}\lambda_{i,p})\right)=0,
\end{align*}
because the first and second terms cancel each other, as can be seen by interchanging the variables $s$ and $i$ in the second sum. Similarly for the term $\diamond$, we have
\begin{align*}
\sum_{s=1}^n\left(\frac{\partial{\mathscr H}_{\varnothing}}{\partial a_{s,n-1}}\frac{\partial{\mathscr H}_{\mathcal I}^Z}{\partial\lambda_{s,n-1}}-\frac{\partial{\mathscr H}_{\varnothing}}{\partial\lambda_{s,n-1}}\frac{\partial{\mathscr H}_{\mathcal I}^Z}{\partial a_{s,n-1}}\right)&
=F(a;\lambda)\left(\sum_{s=1}^n\sum_{{\substack{p,q=1\\p\neq q}}}^n\sum_{i=1}^na_{s,n-1}\lambda_{i,n-1}(a_{i,q}\lambda_{s,q}+a_{i,p}\lambda_{s,p})\right.\\
&\left.-\sum_{s=1}^n\sum_{{\substack{p,q=1\\p\neq q}}}^n\sum_{i=1}^na_{i,n-1}\lambda_{s,n-1}(a_{s,q}\lambda_{i,q}+a_{s,p}\lambda_{i,p})\right)=0,
\end{align*}
which vanishes using exactly the same argument as before.
\end{proof}

\begin{Ob}
As seen in \cite{GG}, a commuting relation for Hamiltonians can be successfully used to compute geodesics. The calculations in the present situation seem to be quite involved to be of interest. In future work, we will explore simplifications of the full Hamiltonian equations in order to find interesting curves in ${\rm SL}(n,{\mathbb R})$ or other real simple Lie groups.
\end{Ob}

\subsection{Solutions of \eqref{eq:Hamsys} for $n=2$}\label{ssec:n=2}

As expected, the Hamiltonian \eqref{eq:Hamcoord} takes a particularly simple form for $n=2$, namely
\[
{\mathscr H}(a;\lambda)=a_{11}a_{12}\lambda_{11}\lambda_{12}+a_{11}a_{22}\lambda_{12}\lambda_{21}+a_{12}a_{21}\lambda_{11}\lambda_{22}+a_{21}a_{22}\lambda_{21}\lambda_{22},
\]
and thus the differential equations in \eqref{eq:Hamsys} can be written as
\begin{equation}\label{eq:system_n=2}
\begin{cases}
\dot a_{11}=a_{12} M,\\
\dot a_{12}=a_{11} N,\\
\dot a_{21}=a_{22} M,\\
\dot a_{22}=a_{21} N,
\end{cases}\qquad
\begin{cases}
\dot\lambda_{11}=-\lambda_{12} N,\\
\dot\lambda_{12}=-\lambda_{11} M,\\
\dot\lambda_{21}=-\lambda_{22} N,\\
\dot\lambda_{22}=-\lambda_{21} M,
\end{cases}
\end{equation}
where $M=a_{11}\lambda_{12}+a_{21}\lambda_{22}$ and $N=a_{12}\lambda_{11}+a_{22}\lambda_{21}$.

The following five conserved quantities are obtained as immediate consequences of the system of differential equations \eqref{eq:system_n=2}
\begin{align}
&\begin{cases}
\dot a_{11}\lambda_{11}+a_{12}\dot\lambda_{12}=0\\
\dot a_{12}\lambda_{12}+a_{11}\dot\lambda_{11}=0
\end{cases}\implies
a_{11}\lambda_{11}+a_{12}\lambda_{12}=C_1,\label{eq:conserved1}\\
&\begin{cases}
\dot a_{21}\lambda_{11}+a_{22}\dot\lambda_{12}=0\\
\dot a_{22}\lambda_{12}+a_{21}\dot\lambda_{11}=0
\end{cases}\implies 
a_{21}\lambda_{11}+a_{22}\lambda_{12}=C_2,\label{eq:conserved2}\\
&\begin{cases}
\dot a_{11}\lambda_{21}+a_{12}\dot\lambda_{22}=0\\
\dot a_{12}\lambda_{22}+a_{11}\dot\lambda_{21}=0
\end{cases}\implies
a_{11}\lambda_{21}+a_{12}\lambda_{22}=C_3,\label{eq:conserved3}\\
&\begin{cases}
\dot a_{21}\lambda_{21}+a_{22}\dot\lambda_{22}=0\\
\dot a_{22}\lambda_{22}+a_{21}\dot\lambda_{21}=0
\end{cases}\implies
a_{21}\lambda_{21}+a_{22}\lambda_{22}=C_4,\label{eq:conserved4}\\
&\begin{cases}
\dot M=(a_{12}\lambda_{12}-a_{11}\lambda_{11}+a_{22}\lambda_{22}-a_{21}\lambda_{21})M\\
\dot N=-(a_{12}\lambda_{12}-a_{11}\lambda_{11}+a_{22}\lambda_{22}-a_{21}\lambda_{21})N
\end{cases}\implies
MN=C_5,\label{eq:conserved5}
\end{align}
where $C_1,\dots,C_5\in{\mathbb R}$ are constants. Note that the conserved quantities \eqref{eq:conserved1}--\eqref{eq:conserved4} become the linear system of equations
\[
\begin{pmatrix}a_{11}&a_{12}&0&0\\a_{21}&a_{22}&0&0\\0&0&a_{11}&a_{12}\\0&0&a_{21}&a_{22}\end{pmatrix}\begin{pmatrix}\lambda_{11}\\\lambda_{12}\\\lambda_{21}\\\lambda_{22}\end{pmatrix}=\begin{pmatrix}C_1\\C_2\\C_3\\C_4\end{pmatrix}
\]
which can be solved explicitly to obtain
\begin{equation}\label{eq:momenta}
(\lambda_{11},\lambda_{12},\lambda_{21},\lambda_{22})=\big(C_1a_{22}-C_2a_{12},-C_1a_{21}+C_2a_{11},C_3a_{22}-C_4a_{12},-C_3a_{21}+C_4a_{11}\big),
\end{equation}
taking into consideration that $\begin{pmatrix}a_{11}&a_{12}\\a_{21}&a_{22}\end{pmatrix}\in {\rm SL}(2,{\mathbb R})$. Replacing equality \eqref{eq:momenta} in the expressions for $M$ and $N$, we can decouple the system of differential equations \eqref{eq:system_n=2} to describe the Hamiltonian trajectories as
\begin{equation}\label{eq:systSL}
\begin{cases}
\dot a_{11}=a_{12} \big((C_4-C_1)a_{11}a_{21}+C_2a_{11}^2-C_3a_{21}^2\big),\\
\dot a_{12}=a_{11} \big((C_1-C_4)a_{12}a_{22}-C_2a_{12}^2+C_3a_{22}^2\big),\\
\dot a_{21}=a_{22} \big((C_4-C_1)a_{11}a_{21}+C_2a_{11}^2-C_3a_{21}^2\big),\\
\dot a_{22}=a_{21} \big((C_1-C_4)a_{12}a_{22}-C_2a_{12}^2+C_3a_{22}^2\big).
\end{cases}
\end{equation}
Solving explicitly this system for all cases seems to be quite complicated. Nevertheless, some special solutions can be obtained for specific values to the constants $C_1,\dots,C_4$. Two well-known families of solutions, both assuming the usual initial condition $\begin{pmatrix}a_{11}(0)&a_{12}(0)\\a_{21}(0)&a_{22}(0)\end{pmatrix}=\begin{pmatrix}1&0\\0&1\end{pmatrix}$ and also $C_1=C_4$, are given by:
\begin{description}
\item[$C_2=C_3=C\neq0$] In this case, the system \eqref{eq:systSL} becomes
\[
\begin{cases}
\dot a_{11}=Ca_{12}(a_{11}^2-a_{21}^2),\\
\dot a_{12}=-Ca_{11}(a_{12}^2-a_{22}^2),\\
\dot a_{21}=Ca_{22} (a_{11}^2-a_{21}^2),\\
\dot a_{22}=-Ca_{21} (a_{12}^2-a_{22}^2),
\end{cases}
\]
whose solutions can be found by inspection
\[
\begin{pmatrix}a_{11}(t)&a_{12}(t)\\a_{21}(t)&a_{22}(t)\end{pmatrix}=\begin{pmatrix}\cosh(Ct)&\sinh(Ct)\\\sinh(Ct)&\cosh(Ct)\end{pmatrix}.
\]
\item[$C_2=-C_3=C\neq0$] Similarly, the system \eqref{eq:systSL} becomes
\[
\begin{cases}
\dot a_{11}=Ca_{12}(a_{11}^2+a_{21}^2),\\
\dot a_{12}=-Ca_{11}(a_{12}^2+a_{22}^2),\\
\dot a_{21}=Ca_{22} (a_{11}^2+a_{21}^2),\\
\dot a_{22}=-Ca_{21} (a_{12}^2+a_{22}^2),
\end{cases}
\]
whose solutions can be found by inspection
\[
\begin{pmatrix}a_{11}(t)&a_{12}(t)\\a_{21}(t)&a_{22}(t)\end{pmatrix}=\begin{pmatrix}\cos(Ct)&-\sin(Ct)\\\sin(Ct)&\cos(Ct)\end{pmatrix}.
\]
\end{description}

Besides the previous non-trivial solutions, it occurs that when $C_2=C_3=0$ the only solution is constant. To see this, rename $C=C_4-C_1$ and thus the system \eqref{eq:systSL} is
\[
\begin{cases}
\dot a_{11}=Ca_{12}a_{11}a_{21},\\
\dot a_{12}=-Ca_{11}a_{12}a_{22},\\
\dot a_{21}=Ca_{22}a_{11}a_{21},\\
\dot a_{22}=-Ca_{21}a_{12}a_{22}.
\end{cases}\implies
\begin{cases}
a_{22}\dot a_{11}+a_{11}\dot a_{22}=0,\\a_{21}\dot a_{12}+a_{12}\dot a_{12}=0,
\end{cases}
\]
As a consequence of the initial condition, it follows that
\[
\begin{cases}
a_{11}a_{22}=1,\\a_{12}a_{21}=0,
\end{cases}\implies
\begin{cases}
\dot a_{11}=0,\\
\dot a_{12}=-Ca_{12},\\
\dot a_{21}=Ca_{21},\\
\dot a_{22}=0,
\end{cases}\implies
\begin{cases}
a_{11}(t)=a_{22}(t)=1,\\
a_{12}(t)=a_{21}(t)=0.
\end{cases}
\]

\begin{Ob}
A similar study as the present one can be done in any semi-simple real Lie group, which will, in principle, be sub-pseudo-Riemannian whenever the group is not compact. This case-by-case analysis does not seem too appealing for the authors and in the future we will be looking for a more structural approach.
\end{Ob}

\subsection*{Acknowledgments} 

The authors would like to thank Dr. Diego Lagos for several fruitful discussions at the beginning of this project.

\end{document}